\documentclass[a4paper,11pt]{article}
\usepackage[utf8]{inputenc}
\usepackage{amsmath, amssymb, amsthm}
\usepackage[margin=2cm]{geometry}
\usepackage{hyperref}

\newcounter{mnotecount}

\newcommand{\mnotex}[1]
{\protect{\stepcounter{mnotecount}}$^{\mbox{\footnotesize $\bullet$\themnotecount}}$
\marginpar{
\raggedright\tiny\em
$\!\!\!\!\!\!\,\bullet$\themnotecount: #1} }

\newcommand{\D}{\mathcal{D}}

\newtheorem{theorem}{Theorem}[section]
\newtheorem{lemma}[theorem]{Lemma}
\newtheorem{proposition}[theorem]{Proposition}

\theoremstyle{definition}
\newtheorem{definition}[theorem]{Definition}

\newtheorem{remark}[theorem]{Remark}
\title{}
\author{}

\begin{document}

\title{Conformal Einstein spaces and conformally covariant operators}

\author{ Alfonso Garc\'\i a-Parrado\footnote{e-mail: agparrado@uco.es},
	Jónatan Herrera\footnote{e-mail: jherrera@uco.es}
	and Miguel Vadillo\footnote{f92vacam@uco.es} \\
	Departamento de Matem\'aticas, Universidad de C\'ordoba,\\
	Campus de Rabanales, 14071, C\'ordoba, Spain\\
}
\maketitle

\begin{abstract}
In this article we give general neccessary and sufficient conditions
to ensure that a pseudo-Riemannian manifold is conformal to an Einstein
space. These conditions are algorithmic in \emph{the metric tensor}
whenever the Weyl endomorphism is invertible. Our conditions
depend in an essential manner on the $\mathcal{C}$-connection.
We also show how to construct \emph{conformally covariant, pseudo-differential}
operators which has an independent interest.
\end{abstract}

\maketitle

\section{Introduction}

Conformal structures play a fundamental role in both Mathematics and Physics. In Mathematics, for instance, conformal structures are essential in various fields, including complex analysis, computation, and computer graphics. In Physics, fundamental interactions—such as Maxwell's equations—and the causal structure of spacetime depend strictly on the underlying conformal structure of the model.

As it is well known, unlike a metric structure, a conformal structure associates to each point of the manifold a family of metrics, all related by a conformal factor. Although these metrics share certain characteristics, such as angle measurement, they differ in other elements. This is the case for the scalar curvature, where the relationship between $R[g]$ and $R[\tilde{g}]$ (with $g_{ab} = \Omega^2 \tilde{g}_{ab}$) is given by the following expression in dimension $n$:

\begin{equation}
\label{eq:main:1}
R[g] = \Omega^{-2} \left( R[\tilde{g}] - 2(n-1)\Delta \ln \Omega - (n-1)(n-2)|\nabla \ln \Omega|^2 \right).
\end{equation}

We say that a tensor $T[g]$ is conformally covariant with weight $s\in\mathbb{R}$ if it satisfies the relation $T[\tilde{g}]=\Omega^s T[{g}]$. The case $s=0$ is called conformally invariant. Observe that the metric tensor itself is covariant with weight $-2$, while other tensors, such as the Weyl tensor, are conformally invariant ($s=0$).

A classical problem in conformal geometry is the \emph{curvature prescription problem}; that is, determining whether there exists a metric within the conformal class with a specific curvature. The classical example of such problems is the so-called \emph{Yamabe problem} \cite{yamabe1960deformation} (see also \cite{brendle2010recentprogressyamabeproblem}), which studies the existence of a metric conformal to a given one with constant scalar curvature. This problem allow us to determine for instance if a certain surface shares the shape of one of the model spaces: the sphere, the plane, or the hyperbolic space. The study of this problem relies on the conformal Laplacian (or Yamabe operator), a modification of the standard Laplacian that is conformally covariant. In the compact case, the Yamabe problem in dimension $2$ is consequence of the uniformization problem. In dimension $D\geq 3$ the problem was solved in several steps involving works of Aubin \cite{Aubin_1986}, Trudinger \cite{trudinger1968} and Schoen \cite{Schoen_1984}.

Observe however that for dimensions bigger than $2$, both this operator and the scalar curvature lose some of their geometric and topological meaning. In fact, the Gauss-Bonnet theorem, which links the geometry given by the scalar curvature to the topology of the surface, is no longer valid, becoming necessary to consider other curvatures and operators.

In dimension $4$, the Paneitz operator \cite{Paneitz2008} and the so-called $Q$-curvature \cite{Branson_1991} appear naturally. The latter plays a role analogous to the scalar curvature in dimension $2$, providing a Gauss-Bonnet type result in dimension $4$ \cite{chang_what_2008}. Furthermore, both arise naturally in different areas of physics, such as quantum field theory (where the operator arises naturally in the study of the trace anomaly of the stress-energy tensor for a scalar field \cite{Fradkin_1982}) or the AdS/CFT correspondence. Additionally, in dimension $4$, the Paneitz operator transforms under conformal changes just as the Laplacian does in dimension $2$, becoming crucial for solving the $4$-dimensional analogue of the Yamabe problem: prescribing the $Q$-curvature.

This framework was generalized to arbitrary even dimensions thanks to the work of Graham, Jenne, Mason, and Sparling \cite{GJMSOperator1992}. In their work, they constructed a family of operators, known as \emph{GJMS operators}, which are conformal operators of appropriate weights. Similarly, the so-called $Q_{2k}$-curvatures are constructed, which generalize both the $Q$-curvature and the scalar curvature, thus providing a unified framework for these studies.

\smallskip

Our work follows a similar line of research. However, our goal is to determine if a spacetime is \emph{conformally Einstein}; that
is, if there exists a metric within the conformal class that is an Einstein metric. This problem has a long-standing history,
dating back to Brinkmann \cite{brinkmann_riemann_1924,brinkmann_einstein_1925}, Schouten \cite{schouten_ricci-calculus_1954} and
Szekeres \cite{szekeres_spaces_1963}. In the 80's, Kozameh, Newman and Tod \cite{kozameh_conformal_1985} provide a set of
conditions ensuring when a $4$-dimensional Lorentzian manifold is conformally Einstein. In 2006, Gover and Nurowski
\cite{gover_obstructions_2006} generalize such set of conditions to arbitrary dimension.

Our aim is to approach this problem from previous perspective of a curvature prescription problem, so we provide a set of conditions more amenable from a computational standpoint. To this end, we will consider two main ingredients. On the one hand, we will substitute curvatures with the more general notion of metric concomitant, that is, a quantity obtained in general in terms of the metric tensor, its inverse and the Levi-Civita connection. As it is clear, any curvature mentioned before is in fact a metric concomitant. On the other hand, we will make use of the so-called $\mathcal{C}$-connection. This connection has already been used by one of the authors to characterize Conformally Einstein spacetimes in \cite{garciaparrado2024einstein}. In that previous work, the author considered the $4$-dimensional case and successfully characterized such spaces under the assumption that the Weyl tensor was non-degenerate. The goal of the present paper is to remove both restrictions, providing a fully general result and a systematic, computational methodology to determine such spacetimes.

\medskip

The paper is organized as follows. Section~\ref{prelconfgeo} introduces the foundational framework used throughout the paper, as well as key concepts such as conformally covariant metric concomitants. Section~\ref{confcovdiff} is devoted to conformally covariant differentiation. Here, we establish Lemma~\ref{lemm:fund}, which relates the transition tensor between the Levi-Civita connections of two conformal metrics to their corresponding Schouten and Weyl tensors. Furthermore, we develop a general notion of conformal differentiation for conformally covariant metric concomitants.

In Section~\ref{cconnection}, we specialize this framework to conformally invariant concomitants, defining the notion of $\mathcal{C}$-connection. This connection serves as the fundamental tool for our characterization of conformally Einstein spacetimes. Section~\ref{charEinstein} presents the main results of the paper, Theorems~\ref{thm:main1} and \ref{thm:main2}. These theorems provide the necessary and sufficient conditions to determine whether a spacetime is conformally Einstein, covering cases where the Weyl endomorphism is both invertible and non-invertible. Finally, these conditions are applied in Section~\ref{applications} to two scenarios of interest: the characterization of conformally flat spacetimes and the analysis of conformally Einstein Brinkmann spacetimes.

\section{Notions of conformal geometry}\label{prelconfgeo}

Let $\mathcal{M}$ be a smooth $D$-dimensional manifold with $D\geq 3$
and define on $\mathcal{M}$ two pseudo-Riemannian metrics: $\tilde{g}
_{ab}$ (which we will name physical metric), and $g_{ab}$ (which will be
named unphysical metric). These two pseudo-Riemannian spaces $
(\mathcal{M},\tilde{g}_{ab})$ and $(\mathcal{M},g_{ab})$ are said to be
conformally related if they follow the relation

\begin{equation}
	g_{ab}=\Omega^2\tilde{g}_{ab}
	\label{conformal relation}
\end{equation}

The conformal factor $\Omega^2$ is assumed to be a smooth positive function which does not vanish in $\mathcal{M}$.
The conformal relation (\ref{conformal relation}) states that the spaces $(\mathcal{M},\tilde{g}_{ab})$ and $
(\mathcal{M},g_{ab})$ have the same \emph{conformal structure}.
Also, this equation gives us the explicit relation between the
contravariant metric tensors

\begin{equation}
	g^{ab}=\frac{\tilde{g}^{ab}}{\Omega^2}
	\label{inverse conformal relation}
\end{equation}

Let us denote by $T(\mathcal{M})^r_s$ the bundle of
$r$-contravariant and $s$-covariant tensors defined in the standard way. Also abstract index
notation with small latin indices will be used throughout.
Indices enclosed within round or square brackets denote symmetrization or
antisymmetrization respectively.
For all tensors, indices will always be raised and lowered using the unphysical
metric $g_{ab}$, except for the physical metric, $\tilde{g}_{ab}$, for which the usual notation
for its inverse metric
will be used
\begin{equation}
	\tilde{g}^{ab}\tilde{g}_{ac}=\delta^b_c.
	\label{unphysical kronecker delta}
\end{equation}
$\nabla_a$ and $\tilde{\nabla}_a$ are the respective torsion-free Levi-Civita
connections of $g_{ab}$ and $\tilde{g}_{ab}$. From them one defines
the curvature tensors in the standard fashion.
Our conventions for the (unphysical) Riemann, Ricci and Weyl tensors are

\begin{equation}
	\nabla_a \nabla_b \omega_c - \nabla_b \nabla_a \omega_c= R_{abc}{}^{d}{} \omega_d,
	\label{riemann tensor definition}
\end{equation}
\begin{equation} 
	R_{ac} \equiv R_{abc}{}^{b}{},
	\label{ricci tensor definition}
\end{equation}
\begin{equation}
	C_{abc}{}^{d}{} \equiv R_{abc}{}^{d}{} - 2 L_{[b}{}^{d}{} g_{a]c} - 2 \delta_{[b}{}^{d}{} L_{a]c},
	\label{weyl tensor definition}
\end{equation}

where $L_{ab}$ is the unphysical Schouten tensor, defined by
\begin{equation}
	L_{ab} \equiv \frac{R_{ab}}{D-2} - \frac{g_{ab}R}{2(D-1)(D-2)},
	\label{schouten tensor definition}
\end{equation}
and $R=g^{ab}R_{ab}$ is the Ricci scalar. Also, as we did
before with the metric and the Levi-Civita connection,
tensors defined in terms of the physical metric $
\tilde{g}_{ab}$ will be denoted with a tilde over the
symbol employed for the unphysical spacetime tensor. The
difference between the Levi-Civita connection
coefficients of the physical and the unphysical metrics
is called the transition tensor $\Gamma[\tilde{\nabla},
\nabla]^l{}_{ac}$ given by
\begin{equation}
	\Gamma[\nabla, \tilde{\nabla}]^{l}{}_{ac} = 2 \delta^{l}{}_{(a} \Upsilon_{c)} - g_{ac} g^{lr}\Upsilon_r,
	\label{eq:transition-tensor}
\end{equation}

where we defined

\begin{equation}
	\Upsilon_a \equiv \frac{\nabla_a \Omega}{\Omega}.
	\label{upsilon definition}
\end{equation}

More information about this transition tensor can be found in
\cite{juanbook}.
Now, we have all the information necessary to
make the introduction of the concept of conformally covariant metric
concomitant,
which will be essential for the development of this work.
\begin{definition}
	\emph{Metric concomitant:}
	A metric concomitant is a quantity defined in terms of the metric tensor, its inverse and the Levi-Civita covariant derivative.
	\label{metric concomitant definition}
\end{definition}

	Some examples of metric concomitants that have already been introduced are the Riemann curvature of the Levi-Civita
	connection (\ref{riemann tensor definition}), as well as its metric contractions (Ricci tensor, (\ref{ricci tensor
	definition}), and Ricci scalar), the Weyl tensor (\ref{weyl tensor definition}) and Schouten tensor (\ref{schouten tensor
	definition}). Of course, the unphysical metric concomitants always admits its counterparts called physical metric
	concomitants denoted with tilde on top of the tensor signature.\\

\begin{definition}
	\emph{Conformally covariant metric concomitant:}
	A metric concomitant $T(g, \nabla)$ is said to be conformally covariant if under the transformation given by
	(\ref{conformal relation}) one has the property
	\begin{equation}
		\tilde{T}(\tilde{g}_{ab}, \tilde{\nabla}) = \Omega^{s}T(g_{ab}, \nabla), \quad s \in \mathbb{R},
		\label{conformally covariant metric concomitant relation}	
	\end{equation}
	where $s\in \mathbb{R}$
	is called the conformal weight of $T$.
	\label{conformally covariant metric concomitant}
\end{definition}
For example the Weyl tensor $C_{abc}{}^{d}{}$ is conformally covariant with weight zero
(aka conformally invariant) because $\tilde{C}_{abc}{}^{d}{}=C_{abc}{}^{d}{}$.
Further details about the previous definitions can be found in
\cite{juanbook,branson1989conformal}.

\section{Conformally covariant differentiation}\label{confcovdiff}

Our aim in this section is to introduce a notion of covariant differentiation which is compatible with the conformal structure. For this, we will generalize the results obtained in \cite{garciaparrado2024einstein}, which were valid only for $4$-dimensional spacetimes with invertible Weyl endomorphism.

\subsection{The fundamental lemma}

Let us start from the relation between the unphysical
and physical Schouten tensors, ($L_{ab}$ and $\tilde{L}_{ab}$ respectively)
which can be written in the form
\begin{equation}
	L_{ab}=\tilde{L}_{ab}-\Upsilon_{a}\Upsilon_{b}-\nabla_{a}\Upsilon_{b}+\frac{1}{2}g_{ab}\Upsilon_{c}\Upsilon^{c}.
	\label{eq:schouten-tensors-relation}
\end{equation}
Additional information of this formula can be found in \cite{juanbook}.
If we compute the unphysical covariant derivative of this expression we can write

\begin{equation}
	2\nabla_{[m}L_{a]b}=2\nabla_{[m}\tilde{L}_{a]b}-2\nabla_{[m}\Upsilon_{a]}\Upsilon_{b}-2\nabla_{[m}\nabla_{a]}\Upsilon_{b}.
\end{equation}
Knowing that $2\nabla_{[m}\nabla_{a]}\omega_{b}=\nabla_{m}\nabla_{a}\omega_{b}-\nabla_{a}\nabla_{m}\omega_{b}=R_{mab}{}^{d}{}\omega_{d}$, and the definition of the Weyl tensor (\ref{weyl tensor definition}), we can rewrite the previous equation as\\

\begin{equation}
	2 \tilde{L}_{b[p} g_{m]a} \, \Upsilon^b + 2 \tilde{L}_{a[m} \, \Upsilon_{p]} + \Upsilon^b C_{abmp} + 2 \nabla_{[p} \tilde{L}_{m]a} + 2 \nabla_{[m} L_{p]a} = 0.
\end{equation}

Now, we use (\ref{eq:transition-tensor}) to express $\nabla$ in terms of $\tilde{\nabla}$,
getting after some algebra

\begin{equation}
	\tilde{\nabla}_{[b} \tilde{L}_{e]a} = \frac{1}{2} \Upsilon^c C_{acbe} + \nabla_{[b} L_{e]a}.
	\label{integrability condition}
\end{equation}

\noindent
The previous
equation can be written in the form
\begin{equation}
	\Upsilon_c C^{ac}{}_{be}{}=2\left(\tilde{\nabla}_{[b}\tilde{L}_{e]}{}^{a}
-\nabla_{[b}L_{e]}{}^{a}\right).
	\label{eq:equation-upsilon}
\end{equation}
We are going to use this equation to isolate $\Upsilon_c$. For that, we regard
$C^{ac}{}_{be}$ as an field of endomorphisms in
End($\Omega^2(\mathcal{M})$)
($\Omega^2(\mathcal{M})$ is
set of 2-vector fields over the manifold $M$). Thus
\begin{equation}
C^{ac}{}_{be}\psi^{be}=\chi^{ac}.
\label{eq:weyl-endomorphism}
\end{equation}
The endomorphism $C^{ac}{}_{be}$ is known as the \emph{Weyl endomorphism}
A straightforward computation shows that the relation between the
physical Weyl endomorphism and the unphysical one is
\begin{equation}
 \tilde{C}^{ac}{}_{be}=\Omega^2C^{ac}{}_{be}{}.
 \label{eq:weyl-endomorphism-physical-to-unphys}
\end{equation}
The Weyl endomorphism
can be either invertible or not. If it is invertible we shall denote its inverse by
$W^{ac}{}_{be}$ and therefore we have the relation
\begin{equation}
C^{ac}{}_{be}W^{be}{}_{lh}=\delta^{a}{}_{[l}\delta^{c}{}_{h]}.
\label{eq:weyl-inverse}
\end{equation}
If it is not, we can define its \emph{Moore-Penrose pseudoinverse},
denoted by $(W^+)^{ac}{}_{be}$ which is unique (see Appendix~\ref{app:pseudo-inverse}).

Now we are ready to present a lemma,
which will be essential for the computations developed in this article.
\begin{lemma}\label{lemm:fund}
Let
\begin{eqnarray}
&&
\Lambda_q\equiv\frac{4}{1-D} \nabla_{[b}L_{e]p}W^{bep}{}_{q},\quad
 \tilde{\Lambda}_q\equiv\frac{4}{1-D} \tilde{\nabla}_{[b}\tilde{L}_{e]p}\tilde{g}^{pr}\tilde{W}^{be}{}_{rq},\label{eq:define-Lambda}\\
&&
	\Lambda^\xi_q\equiv\frac{4}{1-D} \nabla_{[b}L_{e]p}(W^+)_{q}{}^{bep}+
	\frac{2}{1-D}\left(\delta^{m}{}_{[p}\delta^{n}{}_{q]}-
	(W^+)_{pq}{}^{rs}C^{mn}{}_{rs}
	\right)\xi^p{}_{mn},\label{eq:define-Lambda-plus}\\
&&
	\tilde{\Lambda}^{\tilde\xi}_q\equiv\frac{4}{1-D} \tilde{\nabla}_{[b}\tilde{L}_{e]p}\tilde{g}^{pr}(\tilde{W}^+)_{qr}{}^{be} +
	\frac{2}{1-D}\left(\delta^{m}{}_{[p}\delta^{n}{}_{q]}-
	(\tilde{W}^+)_{pq}{}^{rs}\tilde{C}^{mn}{}_{rs}
	\right)\tilde{\xi}^p{}_{mn},\label{eq:define-Lambda-plus-tilde}
\end{eqnarray}
where $\xi^p{}_{[mn]}=\xi^p{}_{mn}$,
$\tilde{\xi}^p{}_{[mn]}=\tilde{\xi}^p{}_{mn}$ are some tensors and
$\tilde{\xi}^p{}_{mn}$ is arbitrary.
If the Weyl endomorphism is invertible then $\Upsilon_{a}$ defined in
\eqref{upsilon definition} fulfills the relation
\begin{equation}
	\Upsilon_a=\Lambda_a-\tilde{\Lambda}_a,
	\label{eq:upsilon-schouten}
\end{equation}
Otherwise the relation is

\begin{equation}
	\Upsilon_a=\Lambda^\xi_a-\tilde{\Lambda}^{\tilde\xi}_a.
	\label{eq:upsilon-schouten-plus}
\end{equation}
\label{lem:upsilon-a}
\end{lemma}
\begin{proof}
Let us start by the invertible Weyl endomorphism case.
First of all \eqref{eq:weyl-inverse} takes the following form
if we use the physical Weyl endomorphism instead of the unphysical
one
\begin{equation}
\tilde{C}^{ac}{}_{be}\tilde{W}^{be}{}_{lh}=\delta^{a}{}_{[l}\delta^{c}{}_{h]}.
\label{eq:weyl-inverse-physical}
\end{equation}
Since $\tilde{C}^{ac}{}_{be}=\Omega^2C^{ac}{}_{be}{}$ \eqref{eq:weyl-inverse-physical}
entails
\begin{equation}
\Omega^2\tilde{W}^{pq}{}_{cd}=W^{pq}{}_{cd}.
\label{eq:weyl-physical-invertible}
\end{equation}
Next we take \eqref{eq:equation-upsilon} and multiply
both sides by $W^{be}{}_{pq}$, getting
\begin{equation}
	\Upsilon_{c}\!
	\left(\delta_{p}{}^{a}\,\delta^{c}{}_{q}-
	\delta^{c}{}_{p}\,\delta^{a}{}_{q}\right)
	=4\left(\tilde{\nabla}_{[b}\tilde{L}_{e]}{}^{a}
	-\nabla_{[b}L_{e]}{}^{a}\right) W^{be}{}_{pq}.
\end{equation}
Taking the $p-a$ trace we can easily isolate $\Upsilon_q$
\begin{equation}
	\Upsilon_{q}=\frac{4}{D-1}
	\left(\tilde{\nabla}_{[b}\tilde{L}_{e]p}-\nabla_{[b}L_{e]p}\right)
	W^{bep}{}_{q}.
	\label{eq:upsilon-invertible}
\end{equation}
Also we can write \eqref{eq:upsilon-invertible} as
\begin{equation}
\Upsilon_{q}=\frac{4}{D-1}
\tilde{\nabla}_{[b}\tilde{L}_{e]p}W^{bep}{}_{q}
-\frac{4}{D-1}\nabla_{[b}L_{e]p}W^{bep}{}_{q}.
\label{eq:upsilon-invertible-expanded}
\end{equation}
\eqref{eq:upsilon-schouten} follows by using
\eqref{eq:weyl-physical-invertible}
in \eqref{eq:upsilon-invertible-expanded}.
Assume now that the Weyl endomorphism is not invertible. Then,
using the relation $\tilde{C}^{ac}{}_{be}=\Omega^2C^{ac}{}_{be}{}$
and the property $(\lambda A)^+=A^+ / \lambda$, $\lambda\in\mathbb{R}$
of the pseudo-inverse (see Appendix~\ref{app:pseudo-inverse}) one easily concludes that

\begin{equation}
 \Omega^2(\tilde{W}^+)_{cd}{}^{pq}=(W^+)_{cd}{}^{pq}.
\label{eq:weyl-physical-pseudo-inverse}
\end{equation}

Next we use the formula \eqref{eq:pseudo-inverse-indices}
to solve \eqref{eq:equation-upsilon} taking $\Upsilon_q$ as the unkown.
After some manipulations we get

\begin{equation}
	\Upsilon_{q}=\frac{4}{D-1}
	\left(\tilde{\nabla}_{[b}\tilde{L}_{e]p}-\nabla_{[b}L_{e]p}\right)
	(W^+)_{q}{}^{bep}+
	\frac{2}{D-1}\left(\delta^{m}{}_{[p}\delta^{n}{}_{q]}-
	(W^+)_{pq}{}^{rs}C^{mn}{}_{rs}
	\right)\psi^p{}_{mn},
	\label{eq:upsilon-non-invertible}
\end{equation}
where $\psi^p{}_{[mn]}=\psi^p{}_{mn}$ is some tensor.
If we set $\psi^p{}_{mn}=\tilde{\xi}^p{}_{mn}-\xi^p{}_{mn}$ with
$\tilde{\xi}^p{}_{mn}$ chosen arbitrarily
then we can write \eqref{eq:upsilon-non-invertible} as
\begin{eqnarray}
&&
	\Upsilon_{q}=\frac{4}{D-1}
	\tilde{\nabla}_{[b}\tilde{L}_{e]p}(W^+)_{q}{}^{bep}-
	\frac{4}{D-1}\nabla_{[b}L_{e]p}
	(W^+)_{q}{}^{bep}+\nonumber\\
&&
	\frac{2}{D-1}\left(\delta^{m}{}_{[p}\delta^{n}{}_{q]}-
	(\tilde{W}^+)_{pq}{}^{rs}\tilde{C}^{mn}{}_{rs}\right)\tilde{\xi}^p{}_{mn}- \nonumber\\
&& 
	\frac{2}{D-1}\left(\delta^{m}{}_{[p}\delta^{n}{}_{q]}-
	(W^+)_{pq}{}^{rs}C^{mn}{}_{rs}\right)\xi^p{}_{mn},
	\label{eq:upsilon-non-invertible-expanded}
\end{eqnarray}
where we also used
$(\tilde{W}^+)_{pq}{}^{rs}\tilde{C}^{mn}{}_{rs}=(W^+)_{pq}{}^{rs}C^{mn}{}_{rs}$
which is a consequence of \eqref{eq:weyl-physical-pseudo-inverse} and
\eqref{eq:weyl-endomorphism-physical-to-unphys}.
Finally by
defining $\Lambda^{\xi}_{q}$ and $\tilde{\Lambda}^{\xi}_{q}$ as done in \eqref{eq:define-Lambda-plus} and
\eqref{eq:define-Lambda-plus-tilde} we obtain \eqref{eq:upsilon-schouten-plus}.

\end{proof}

\begin{remark}
If the Weyl endomorphism is not invertible then
the compatibility condition of
Theorem \ref{theo:compatibility-condition} applied to
\eqref{eq:equation-upsilon}
reads

\begin{equation}
C_{be}{}^{qp}\,(W^{+})_{qp}{}^{rs}
\left(\tilde{\nabla}_{[r} \tilde{L}_{s]a}-
\nabla_{[r} L_{s]a}\right)
=  \tilde{\nabla}_{[b} \tilde{L}_{e]a}
- \nabla_{[b} L_{e]a}.
\label{eq:compatibility-condition}
\end{equation}
This is a necessary condition for
\eqref{eq:equation-upsilon} to have solutions.
Therefore if the Weyl endomorphism is not invertible
then eq. \eqref{eq:compatibility-condition} has to be appended
to \eqref{eq:upsilon-schouten-plus} as a subsidiary condition.
\end{remark}

\subsection{Conformal differentiation of conformally covariant
metric concomitants}
\label{subsec:conformal-diff}
With Lemma \ref{lem:upsilon-a} established, we are in a position to introduce a new differential operator compatible with conformally covariant metric concomitants; that is, an operator that transforms conformally covariant concomitants into conformally covariant concomitants. First of all, let us
consider the case of a conformally covariant scalar function $\tilde{u} = \Omega^s u$, with $
u \in C^\infty(\mathcal{M})$ and $s\in \mathbb{R}$.

Taking the physical
covariant derivative $\tilde{\nabla} $ of $\tilde{u}$ and replacing it
by its unphysical counterpart $\nabla$ we get
\begin{equation}
	\tilde{\nabla}_a \tilde{u} = s \Omega^{s-1} u \nabla_a \Omega +
	\Omega^s \nabla_a u=\Omega^s\left( su\Upsilon_{a}+\nabla_{a}u
	\right),
	\label{eq:nabla-u}
\end{equation}
where we used the definition of $\Upsilon_{a}$ given by (\ref{upsilon
definition}). As we can see, due to the use of the Leibnitz rule, this
equation shows that the Levi-Civita connection is not conformally
covariant because of the presence of the term $su\Upsilon_a$.
However, taking \eqref{eq:nabla-u} as the starting point, we
can construct other operators which are conformally
covariant.

\begin{theorem}
	Assume that the Weyl endomorphism is invertible and $u$ is a conformally
	covariant scalar with weight $s$. Then
	the quantity $D_a^s u$ with
	$ D_a^s : C^\infty(\mathcal{M}) \to
	C^\infty(\mathcal{M})$, defined by
	\begin{equation}
	D_a^s u \equiv \nabla_a u + s\Lambda_a u,
	\end{equation}
	is conformally covariant of weight $s$ when acting on scalars.
	\label{theo:operator-Da-theorem}
\end{theorem}

\begin{proof}
Using Lemma \ref{lem:upsilon-a} in \eqref{eq:nabla-u} we easily get
	\[
	\quad \tilde{\nabla}_a \tilde{u} = s \Omega^s (\Lambda_a-\tilde{\Lambda}_a) u + \Omega^s \nabla_a u \Leftrightarrow
	\]
	\[
	\quad \tilde{\nabla}_a \tilde{u} + s \Omega^s \widetilde{\Lambda}_a u = \Omega^s (\nabla_a u + s \Lambda_a u) \Leftrightarrow 
	\]
	\[
	\tilde{D}_a^s \tilde{u}=\Omega^s D_a^s u.
	\]
\end{proof}

The operator \(D_a^s: C^\infty(\mathcal{M}) \to C^\infty(\mathcal{M})\)
is linear and it fulfills a \emph{generalized Leibniz rule}
	\begin{equation}\label{eq:generalized-leibnitz}
		D_a^s(w_1 w_2)=(D_a^{s_1}w_1)\,w_2 + w_1\,(D_a^{s_2}w_2),\qquad s=s_1+s_2,\quad \forall w_1,w_2\in C^\infty(\mathcal{M}).
		\tag{23}
	\end{equation}
(see \cite{garciaparrado2024einstein}).

\begin{remark}
 If the Weyl endomorphism is not invertible then
 one can replace $\tilde{\Lambda}_a$ and $\Lambda_a$ by
 $\tilde{\Lambda}^{\tilde\xi}_a$ and $\Lambda^\xi_a$ respectively and carry out a similar proof.
 In this case the corresponding operator is defined by
 \begin{equation}
  (D^\xi)_a^s u \equiv \nabla_a u + s\Lambda^\xi_a u.
 \end{equation}
All the statements about $D^s_a$ can be formulated for $(D^\xi)^s_a$ in a
similar fashion.
\label{rem:degenerate-weyl}
\end{remark}

Theorem \ref{theo:operator-Da-theorem} can be generalized to tensors as follows.
\begin{theorem}
Assume that the Weyl endomorphism is invertible, and define the conformal pseudo-differential operator $D_a^{s,(p,q)} : \mathcal{T}^p_q(\mathcal{M}) \rightarrow \mathcal{T}^{p+1}_q(\mathcal{M}) $, $p,q\in \mathbb{N}$,
$(p, q) \neq (0, 0)$, given by
	\begin{eqnarray}
	&&
	D_a^{s,(p,q)} K^{a_1 \ldots a_p}{}_{b_1 \ldots b_q} \equiv \nabla_a K^{a_1 \ldots a_p}{}_{b_1 \ldots b_q} +\nonumber\\
	&&
	\sum_{j=1}^{p} \overline{M}^{a_j}{}_{ac}(\Lambda, \mathbf{g}, p, s) K^{a_1 \ldots a_{j-1} c a_{j+1} \ldots a_p}{}_{b_1 \ldots b_q} + \nonumber\\
	&&
	\sum_{j=1}^{q} M^{c}{}_{ab_j}(\Lambda, \mathbf{g}, q, s) K^{a_1 \ldots a_p}{}_{b_1 \ldots b_{j-1} c b_{j+1} \ldots b_q},
		\label{eq:conf-diff}
	\end{eqnarray}
	 where
	\begin{eqnarray}
	&&
	M^c{}_{ba}(\Lambda, \mathbf{g}, q, s) \equiv \left( \frac{s + q}{q} \right) \Lambda_b \delta^c{}_a +
	\Lambda_a \delta^c{}_b - g_{ab} g^{cd} \Lambda_d, \label{eq:definition-M}\\
	&&
	\overline{M}^c{}_{ba}(\Lambda, \mathbf{g}, p, s) \equiv \left( \frac{s - p}{p} \right) \Lambda_b \delta^c{}_a - \Lambda_a \delta^c{}_b + g_{ab} g^{cd} \Lambda_d.
	\label{eq:definition-M-bar}
	\end{eqnarray}
If the tensor $K^{a_1 \ldots a_p}{}_{b_1 \ldots b_q}$, is an unphysical metric concomitant conformally covariant of weight $s \in \mathbb{R}$, 
\begin{equation}
\tilde{K}^{a_1 \ldots a_p}{}_{b_1 \ldots b_q} = \Omega^s K^{a_1 \ldots a_p}{}_{b_1 \ldots b_q},
\label{eq:conformally-covariant-tensor}
\end{equation}
then $D_a^{s,(p,q)} K^{a_1 \ldots a_p}{}_{b_1 \ldots b_q}$
is also conformally covariant with the same weight
\begin{equation}
\widetilde{D}_a^{s,(p,q)} \widetilde{K}^{a_1 \ldots a_p}{}_{b_1 \ldots b_q} =
\Omega^s D_a^{s,(p,q)} K^{a_1 \ldots a_p}{}_{b_1 \ldots b_q}.
\end{equation}
\label{th:pseudo-differential-operator}

\end{theorem}
\begin{proof}
	The result can be obtained following a procedure similar to the proof of Theorem 2 in \cite{garciaparrado2024einstein}. For the sake of completeness we review here the main steps.
	
	Consider a metric concomitant that is a conformally covariant tensor of rank $(p,q)$ and conformal
	weight $s$ as shown in \eqref{eq:conformally-covariant-tensor} and take the physical covariant
	derivative of that equation
	\begin{equation}
		\tilde{\nabla}_c \tilde{K}^{a_1 \ldots a_p}{}_{b_1 \ldots b_q}
		= s \Omega^{s-1} \tilde{\nabla}_c \Omega \, K^{a_1 \ldots a_p}{}_{b_1 \ldots b_q}
		+ \Omega^{s} \left(\tilde{\nabla}_c K^{a_1 \ldots a_p}{}_{b_1 \ldots b_q}\right).
	\end{equation}
	Next, use the transition tensor to write the physical derivative in the right hand side
	in terms of the unphysical one getting
	\begin{eqnarray}
	&&
	\tilde{\nabla}_c \tilde{K}^{a_1 \ldots a_p}{}_{b_1 \ldots b_q} =
		s \Omega^{s-1} \nabla_c \Omega \, K^{a_1 \ldots a_p}{}_{b_1 \ldots b_q}+ \Omega^{s} \nabla_c K^{a_1 \ldots a_p}{}_{b_1 \ldots b_q} - \Omega^{s}\sum_{j=1}^{p} \Gamma[\nabla, \tilde{\nabla}]^{q_j}{}_{c b_j} K^{a_1 \ldots a_p}{}_{b_1 \ldots q_{j}\ldots b_q}\nonumber\\
	&&
		+ \Omega^{s}\sum_{j=1}^{q} \Gamma[\nabla, \tilde{\nabla}]^{a_j}{}_{c p_j} K^{a_1 \ldots p_{j}\ldots a_p}{}_{b_1 \ldots b_q}.
	\end{eqnarray}
	
	We can use \eqref{eq:transition-tensor} in combination with Lemma \ref{lem:upsilon-a}
	to write the previous equation in terms of
	$M^{q_j}{}_{c b_j}$ and $\bar{M}^{q_j}{}_{c b_j}$
	as defined by \eqref{eq:definition-M} and \eqref{eq:definition-M-bar} yielding
	\begin{eqnarray}
	&&
	\tilde{\nabla}_c \tilde{K}^{a_1 \ldots a_p}{}_{b_1 \ldots b_q}+ \sum_{j=1}^{p} \Omega^{s} \overline{M}^{a_j}{}_{cq_j}
	(\tilde{\Lambda}, \tilde{\mathbf{g}}, p, s) K^{a_1 \ldots q_j \ldots a_p}{}_{b_1 \ldots b_q}+\sum_{j=1}^{q} \Omega^{s}
	M^{p_j}{}_{cb_j}(\tilde{\Lambda}, \tilde{\mathbf{g}}, q, s) K^{a_1 \ldots a_p}{}_{b_1 \ldots p_j \ldots b_q} = \nonumber\\
	&&
		\Omega^{s}\left(\nabla_c K^{a_1 \ldots a_p}{}_{b_1 \ldots b_q} +\sum_{j=1}^{p} \overline{M}^{a_j}{}_{cq_j}(\Lambda, \mathbf{g}, p, s) K^{a_1 \ldots q_j \ldots a_p}{}_{b_1 \ldots b_q}+\sum_{j=1}^{q} M^{p_j}{}_{cb_j}(\Lambda, \mathbf{g}, q, s) K^{a_1 \ldots a_p}{}_{b_1 \ldots p_j \ldots b_q}\right).\nonumber\\
	&&
	\end{eqnarray}
	
\end{proof}

\begin{remark}
 If the Weyl endomorphism is not invertible then we can proceed in a fashion similar as in
 Remark \ref{rem:degenerate-weyl} and define the operator $(D^\xi)^{s,(p,q)}_a$
 using $\Lambda^\xi_a$.
 \label{rem:weyl-degenerate}
\end{remark}

\section{The $\mathcal{C}$ and $\mathcal{C}^\xi$-connections}\label{cconnection}
If $s=0$ (conformal invariance) then $D^{0,(p,q)}_a$ is actually independent from
$p,q$, as it follows from equations
(\ref{eq:definition-M})-(\ref{eq:definition-M-bar}).
Using \eqref{eq:conf-diff} it is clear that $D^{0,(p,q)}_a$ is a covariant derivative that
defines a new connection called the $\mathcal{C}$-conection \cite{garciaparrado2024einstein}.
\begin{equation}
	\mathcal{C}_a \equiv D^{0,(p,q)}_a.
	\label{eq:c-connection}
\end{equation}

The main properties of the $\mathcal{C}$-connection were studied in
\cite{garciaparrado2024einstein} for the particular case of $D=4$ and an
invertible \emph{Weyl square endomorphism} $C_{ab}$ which is defined as
$C_{ab}\equiv C_{apqr}C_{b}{}^{pqr}{}$. Again if $D=4$ one has the
relation $C_{ab}=(C\cdot C)g_{ab} / 4$, where
$C\cdot C= C_{abcd}C^{abcd}$. The analysis performed
in \cite{garciaparrado2024einstein} generalizes straightforwardly to the
present situation if the Weyl endomorphism is invertible and thus
only the main results will be presented in that case.
The most important properties of the $\mathcal{C}$-connection are that
it is conformally invariant $\tilde{\mathcal{C}}_a = \mathcal{C}_a$
(Theorem 3 of \cite{garciaparrado2024einstein}) and that it is torsionless
as it is deduced from the relation
\begin{equation}
	\Gamma[\mathcal{C}, \nabla]^{c}{}_{ab} = g^{cd} g_{ab} \Lambda_d - \delta_{b}{}^{c}{} \Lambda_a - \delta_{a}{}^{c}{} \Lambda_b,
	\label{eq:connection-C-nabla}
\end{equation}
which in turn is obtained from \eqref{eq:definition-M-bar} by setting $s=0$.
Also, another important property arising from the conformal invariance of the
$\mathcal{C}$-connection is the conformal invariance of
its curvature, $\tilde{\mathcal{R}}_{abc}{}^d=\mathcal{R}_{abc}{}^d$
(see Remark 1 in \cite{garciaparrado2024einstein}).

\begin{proposition}
	The $\mathcal{C}$-connection is a Weyl connection.
	\label{prop:weyl-connection}
\end{proposition}
\begin{proof}
	If we compute the $\mathcal{C}$ covariant derivative of the unphysical
	metric tensor we get
	\begin{equation}
		\mathcal{C}_{a}g_{bc}=\nabla_{a}g_{bc}
		-\Gamma[\mathcal{C},\nabla]^{d}{}_{ab}g_{dc}
		-\Gamma[\mathcal{C}, \nabla]^{d}{}_{ac}g_{bd}.
	\end{equation}
	Using \eqref{eq:connection-C-nabla} in the previous equation
	and the unphysical Levi-Civita condition $\nabla_a g_{bc}=0$ we get
	\begin{equation}
		\mathcal{C}_{a}g_{bc}=2\Lambda_{a}g_{bc},
	\end{equation}
	which is the definition of a Weyl connection.
\end{proof}

If the Weyl endomorphism is not invertible then
we may define a $\mathcal{C}^\xi$-connection by
\begin{equation}
	\mathcal{C}^\xi_a \equiv (D^\xi)^{0,(p,q)}_a.
	\label{eq:c-connection-plus}
\end{equation}
Clearly, all the properties of the $\mathcal{C}$-connection described above
hold \emph{mutatis mutandis} for the $\mathcal{C}^\xi$ connection.

\section{Characterization of conformal Einstein spaces}\label{charEinstein}


In this section, we will make a characterization of conformal $D$-dimensional Einstein spaces, which are the
ones in which $\tilde{\Lambda}_d=0$, or in other words, $\Upsilon_{d}=\Lambda_d$ (Lemma \ref{lem:upsilon-a})
and $\tilde{L}_{ab}=\lambda\tilde{g}_{ab}$.
\begin{theorem}\label{thm:main1}
	Let $(\mathcal{M}, g)$ be a $D$-dimensional pseudo-Riemannian manifold and assume that
	the Weyl endomorphism of its Levi-Civita connection is invertible.
	Then $(\tilde{\mathcal{M}}, \tilde{g}_{ab})$ is an Einstein manifold
	if and only if
	\begin{equation}
		\mathcal{R}_{[ab]} = 0, \quad \mathcal{R}_{(ab)} - \frac{1}{D} g_{ab} \mathcal{R} = 0,
		\label{eq:einstein-manifold}
	\end{equation}
	where $\mathcal{R}_{ab}$ is the Ricci tensor of the $\mathcal{C}_a$ covariant derivative.
	\label{theo:charact-einstein-spaces}
\end{theorem}

\begin{proof}
	 If $(\tilde{\mathcal{M}}, \tilde{g}_{ab})$ is an Einstein manifold then
	 $\tilde{L}_{ab}=\lambda \tilde{g}_{ab}$, $\lambda\in\mathbb{R}$, $\tilde{\Lambda}_a=0$ and
	 $\Upsilon_a=\Lambda_a$ (see Lemma \ref{lem:upsilon-a}). Introducing all this information
	 into \eqref{eq:schouten-tensors-relation}, we get
	\begin{equation}
		\lambda\tilde{g}_{ab}=L_{ab}-\frac{1}{2}g_{ab}
		\Lambda_c\Lambda^c+\nabla_{a}\Lambda_{b}+\Lambda_{a}
		\Lambda_{b},
		\label{eq:equal-zero}
	\end{equation}
	The result follows after expressing the previous condition in terms of the $\mathcal{C}$-covariant
	derivative. Let us work out the details:
	first, we denote the Riemann and Ricci tensors of the $\mathcal{C}$-connection by
	$\mathcal{R}_{abc}{}^{d}{}$, $\mathcal{R}_{ab} \equiv \mathcal{R}_{adb}{}^{d}{}$ respectively.
	Since both $\nabla$ and $\mathcal{C}$ are torsionless,
	the formula that relates the Riemann curvature tensors of $\nabla$ and $\mathcal{C}$ is
	\begin{equation}
		\mathcal{R}_{abc}{}^{d}{} = -2\Gamma[\mathcal{C},\nabla]^{d}{}_{e[a}
		\Gamma[\mathcal{C},\nabla]^{e}{}_{b]c} + R_{abc}{}^{d}{}
		- 2\mathcal{C}_{[a} \Gamma[\mathcal{C}, \nabla]^{d}{}_{b]c}.
		\label{eq:curvature-general-formula}
	\end{equation}
	Using (\ref{eq:connection-C-nabla}) in the previous equation we get
	\begin{eqnarray}
		&&
		\mathcal{R}_{abc}{}^{d}{} = R_{abc}{}^{d}{} + 2\Lambda^d \Lambda_{[a}g_{b]c} + 2\delta_{[a}{}^{d}{} \Lambda_{b]} \Lambda_c +	 2\delta_{[b}{}^{d}{}g_{a]c} \Lambda_e \Lambda^e\nonumber\\
		&&
		+ 2\delta_{[b}{}^{d}{}\mathcal{C}_{a]} \Lambda_c + 2\delta_{c}{}^{d}{} {\mathcal C}_{[a} \Lambda_{b]}
		+ 2g^{de} g_{c[a} \mathcal{C}_{b]} \Lambda_e.
		\label{eq:curvature-c-conection}
	\end{eqnarray}
	From that, we can calculate the relation between the unphysical and the $\mathcal{C}$-Ricci tensors obtaining
	\begin{eqnarray}
		&&
		R_{ac} = \mathcal{R}_{ac} + (D-2)\Lambda_a \Lambda_c + (2-D)(\Lambda_b \Lambda^b) g_{ac} +\nonumber\\
		&&
		(1-D)\mathcal{C}_a \Lambda_c + \mathcal{C}_c \Lambda_a - g_{ac} (g^{db}\mathcal{C}_d \Lambda_b),
		\label{eq:ricci-tensor-c-conection}
	\end{eqnarray}
	and the Ricci and $\mathcal{C}$-Ricci scalars
	\begin{equation}
		R = g^{ac}R_{ac}= \mathcal{R} -(D-1)(D-2)(g^{ab}\Lambda_a \Lambda_b) +2(1-D) (g^{ab}\mathcal{C}_b \Lambda_a).
		\label{eq:ricci-scalar-c-conexion}
	\end{equation}\\
	We have now all the ingredients to compute the Schouten tensor $L_{ab}$ in terms of the $\mathcal{C}$-connection
	using its definition \eqref{schouten tensor definition}. The result is
	\begin{equation}
		L_{ab} =  \frac{\mathcal{R}_{ab}-D\mathcal{C}_a\Lambda_b +2\mathcal{C}_{(a}\Lambda_{b)}}{D-2}
		+\Lambda_a \Lambda_b - \frac{1}{2} \Lambda_c \Lambda^c g_{ab}
		- \frac{\mathcal{R} g_{ab}}{2(D-1)(D-2)}.
		\label{eq:schouten-tensor-c-conection}
	\end{equation}
	From \eqref{eq:ricci-tensor-c-conection} or
	\eqref{eq:schouten-tensor-c-conection} we get
	\begin{equation}
	\mathcal{C}_{[a}\Lambda_{c]} = \frac{1}{D}\mathcal{R}_{[ac]}.
	\label{eq:diff-lambda}
	\end{equation}
	Also we have
	\begin{equation}
		\nabla_b \Lambda_a =\mathcal{C}_b \Lambda_a -\Gamma[\nabla, \mathcal{C}]^{c}{}_{ba} \Lambda_c,
	\end{equation}
	which, by means of (\ref{eq:connection-C-nabla}), becomes
	\begin{equation}
		\nabla_b \Lambda_a = -2 \Lambda_a \Lambda_b + (\Lambda_c \Lambda^c) g_{ab} +\mathcal{C}_b \Lambda_a.
		\label{eq:relation-nabla-c}
	\end{equation}
	Inserting \eqref{eq:ricci-tensor-c-conection}, \eqref{eq:ricci-scalar-c-conexion}, \eqref{eq:relation-nabla-c}
	into \eqref{eq:equal-zero} yields
	\begin{equation}
	 \lambda \tilde{g}_{ab} =
	 \frac{\mathcal{R}_{(ab)}}{D-2}-
		\frac{\mathcal{R} g_{ab}}{2(D-1)(D-2)}.
	\end{equation}
	which leads straight away to the last equation in \eqref{eq:einstein-manifold}. Finally using $\Upsilon_a=\Lambda_a$
	in \eqref{eq:diff-lambda} we obtain the remaining condition in \eqref{eq:einstein-manifold}.

Conversely if \eqref{eq:einstein-manifold} holds then
combining \eqref{eq:einstein-manifold} and \eqref{eq:diff-lambda} we obtain
$\mathcal{C}_{[a}\Lambda_{c]} =\nabla_{[a}\Lambda_{c]}=0$.
Therefore $\Lambda_a$ is locally exact and we have $\Lambda_a=\nabla_a\log\psi=\nabla_a\psi/\psi$ for some function $\psi>0$.
Define a new physical metric tensor by $\tilde{G}_{ab}=g_{ab}/\psi^2$. Equation \eqref{eq:schouten-tensors-relation}
enables us to find the relation between the Schouten tensor $\tilde{S}_{ab}$ of $\tilde{G}_{ab}$ and the
Schouten tensor $L_{ab}$ of $g_{ab}$
\begin{equation}
	L_{ab}=\tilde{S}_{ab}-\upsilon_{a}\upsilon_{b}-\nabla_{a}\upsilon_{b}+\frac{1}{2}g_{ab}\upsilon_{c}\upsilon^{c},\quad
	\upsilon_a\equiv \frac{\nabla_a\psi}{\psi}.
\end{equation}
Setting $\upsilon_{a}=\Lambda_a$ and
inserting \eqref{eq:schouten-tensor-c-conection}, \eqref{eq:relation-nabla-c}
in the previous relation we conclude that
$\tilde{S}_{ab}$ is traceless, which only is possible if $\tilde{G}_{ab}$ is an Einstein metric.
\end{proof}

If the Weyl endomorphism is not invertible then Theorem \ref{theo:charact-einstein-spaces}
still holds if we formulate it in terms of a
$\mathcal{C}^\xi$-connection. For that, we need a preliminary result.
\begin{proposition}
Let $(\mathcal{M}, g)$ be a $D$-dimensional pseudo-Riemannian manifold and assume that
the Weyl endomorphism of its Levi-Civita connection is {\bf not} invertible.
If $(\tilde{\mathcal{M}}, \tilde{g}_{ab})$ is an Einstein manifold then
\begin{equation}
C_{be}{}^{qp}\,(W^{+})_{qp}{}^{rs}
\nabla_{[r} L_{s]a}
= \nabla_{[b} L_{e]a}.
\label{eq:charact-einstein-spaces-sing}
\end{equation}
\label{prop:charact-einstein-spaces-sing}
\end{proposition}

\begin{proof}
If $(\tilde{\mathcal{M}}, \tilde{g}_{ab})$ is an Einstein manifold
then $\tilde{L}_{ab}=\lambda g_{ab}$, $\lambda\in\mathbb{R}$
and if the Weyl endomorphism is not invertible,
\eqref{eq:compatibility-condition} leads straightforwardly
to \eqref{eq:charact-einstein-spaces-sing}.

\end{proof}

\begin{theorem}\label{thm:main2}
Let $(\mathcal{M}, g)$ be a $D$-dimensional pseudo-Riemannian manifold and assume that
the Weyl endomorphism of its Levi-Civita connection is {\bf not} invertible.
Then, under the hypotheses of Proposition \ref{prop:charact-einstein-spaces-sing},
$(\tilde{\mathcal{M}}, \tilde{g}_{ab})$ is an Einstein manifold
if and only if there is a tensor $\xi^a{}_{bc}=\xi^a{}_{[bc]}$ such that
$\Lambda^{\xi}_a$ defined by
	\begin{equation}
		\Lambda^{\xi}_{q} =
		\frac{4}{1-D}\nabla_{[b}L_{e]p}
		(W^+)_{q}{}^{pbe}
		+\frac{2}{1-D}\left(\delta^{m}{}_{[p}\delta^{n}{}_{q]}-
		(W^+)_{pq}{}^{rs}C^{mn}{}_{rs}\right)\xi^p{}_{mn},
		\label{eq:xi-einstein-condition}
	\end{equation}
	is closed and for that $\xi^a{}_{bc}$,
	the $\mathcal{C}^\xi$-connection,
	Ricci tensor $\mathcal{R}^\xi_{ab}$ and Ricci scalar $\mathcal{R}^\xi$
	fulfill the conditions
\begin{equation}
	\mathcal{R}^\xi_{[ab]} = 0, \quad
	\mathcal{R}^\xi_{(ab)} - \frac{1}{D} g_{ab} \mathcal{R}^\xi = 0.
\label{eq:einstein-manifold-plus}
\end{equation}
\label{theo:charact-einstein-spaces-plus}
\end{theorem}
\begin{proof}
	If $(\tilde{\mathcal{M}}, \tilde{g}_{ab})$ is an Einstein manifold then
	$\tilde{L}_{ab}=\lambda \tilde{g}_{ab}$, $\lambda\in\mathbb{R}$, and if we
	pick $\tilde{\xi}^p{}_{mn}$ so that $\left(\delta^{m}{}_{[p}\delta^{n}{}_{q]}-
	(\tilde{W}^+)_{pq}{}^{rs}\tilde{C}^{mn}{}_{rs}
	\right)\tilde{\xi}^p{}_{mn}=0$ in \eqref{eq:define-Lambda-plus-tilde} then
	$\tilde{\Lambda}^{\tilde\xi}_a=0$ and hence
	$\Upsilon_a=\Lambda^{\xi}_a$ for some $\xi^p{}_{mn}=\xi^p{}_{[mn]}$
	(see Lemma \ref{lem:upsilon-a}). Thus $\Lambda^{\xi}_a$ is closed
	as $\Upsilon_a$ is. Now,
	an argument similar to
	the one in the proof of Theorem \ref{theo:charact-einstein-spaces}
	but replacing $\Lambda_a$ by $\Lambda_a^\xi$ and $\mathcal{C}_a$
	by $\mathcal{C}^\xi_a$ leads to
	\eqref{eq:einstein-manifold-plus}.
	The converse's proof proceeds also along same lines as
	in Theorem \ref{theo:charact-einstein-spaces}, again
	replacing $\Lambda_a$ by $\Lambda_a^\xi$ and $\mathcal{C}_a$
	by $\mathcal{C}^\xi_a$ where necessary.
	Finally, $\mathcal{R}^\xi_{[ab]} = 0$ entails
	\begin{equation}
		\mathcal{C}_{[a}\Lambda^{\xi}_{b]} = 0
		\label{eq:xi-constraint}
	\end{equation}
(see \eqref{eq:diff-lambda}) or in other words the 1-form
$\Lambda^{\xi}_{b}$ must be closed as $\mathcal{C}^\xi_a$ is torsionless.
\end{proof}

Theorem \ref{theo:charact-einstein-spaces-plus} provides
a characterization of conformal Einstein spaces that does not depend
only on the unphysical metric, as happens with Theorem \ref{theo:charact-einstein-spaces}
because it contains an additional quantity that is the tensor $\xi^p{}_{mn}$.
However, any choice of $\xi^p{}_{mn}$ fulfilling the conditions of
Theorem \ref{theo:charact-einstein-spaces-plus} would be suitable.
One might not need to solve a differential equation to make such a choice.
See subsection \ref{subsec:ppwave} for an example of this.

\section{Examples}\label{applications}

\subsection{Conformally flat spaces}
Let us assume that the physical manifold
$(\mathcal{M},\tilde{g}_{ab})$ is flat. Then
$(\mathcal{M},g_{ab})$ is conformally flat and both the
physical and unphysical Weyl endomorphisms vanish.
Then, according to the definition given in the appendix \ref{app:pseudo-inverse}
the Weyl endomorphism pseudo-inverse is $(W^+)^{ab}{}_{cd}=0$ and
in that case eqns. \eqref{eq:define-Lambda-plus}-
\eqref{eq:define-Lambda-plus-tilde} yield

\begin{equation}
	\Lambda_{a}^{\xi}=\frac{2\xi^p{}_{pa}}{1-D}, \quad
	\tilde{\Lambda}_{a}^{\tilde \xi}=\frac{2\tilde{\xi}^p{}_{pa}}{1-D},
	\label{eq:lambda-flat-space}
\end{equation}
where
\begin{equation}
\xi^p{}_{pa}=\tilde{\xi}^p{}_{pa}+\frac{(1-D)}{2}\Upsilon_a.
\end{equation}

Using the results found before, we can use the Theorem
\ref{th:pseudo-differential-operator} to find the explicit form of the
pseudo-differential operator $D_c^{s,(p,q)}$.
First, we use \eqref{eq:definition-M} and \eqref{eq:definition-M-bar} to get

\begin{equation}
M^c{}_{ba}(\Lambda, \mathbf{g}, q, s) = \frac{2}{1-D}\left[
\left(\frac{s +q}{q}\right)\xi^{p}{}_{pb}\delta^{c}{}_{a}
+\xi^{p}{}_{pa}\delta^{c}{}_{b}-g_{ab} g^{cd} \xi^p{}_{pd}\right],
\label{eq:M-flat-space}
\end{equation}

\begin{equation}
\overline{M}^c{}_{ba}(\Lambda,\mathbf g,p,s)
=\frac{2}{1-D}\left[
	\left(\frac{s-p}{p}\right)\xi^{p}{}_{pb}\,\delta^{c}{}_{a}
	-\xi^{p}{}_{pa}\,\delta^{c}{}_{b}
	+ g_{ab}\,g^{cd}\,\xi^{p}{}_{pd}
\right].
\label{eq:M-flat-space-bar}
\end{equation}

Next we replace in \eqref{eq:conf-diff} getting

\begin{eqnarray}
&&
D_c^{s,(p,q)} K^{a_1 \ldots a_p}{}_{b_1 \ldots b_q}
\equiv \nabla_c K^{a_1 \ldots a_p}{}_{b_1 \ldots b_q}
+\nonumber\\
&&
\frac{2}{1-D}\sum_{j=1}^{p}\!\left[
\left(\frac{s-p}{p}\right)\xi^{p}{}_{pc}\,\delta^{a_j}{}_{q_j}
-\xi^{p}{}_{pq_j}\,\delta^{a_j}{}_{c}
+g_{q_j c}\,g^{a_j d}\,\xi^{p}{}_{pd}
\right]
K^{a_1 \ldots a_{j-1} q_{j} a_{j+1} \ldots a_p}{}_{b_1 \ldots b_q}
+\nonumber\\
&&
\frac{2}{1-D}\sum_{j=1}^{q}\!\left[
\left(\frac{s+q}{q}\right)\xi^{p}{}_{pc}\,\delta^{r_j}{}_{b_j}
+\xi^{p}{}_{pb_j}\,\delta^{r_j}{}_{c}
-g_{b_j c}\,g^{r_j d}\,\xi^{p}{}_{pd}
\right]
K^{a_1 \ldots a_p}{}_{b_1 \ldots b_{j-1}r_{j}b_{j+1} \ldots b_q}.
\label{eq:pseudo-diff-flat-space}
\end{eqnarray}

Also we can give a formula for the $\mathcal{C}$-connection, which is given by $s=0$, so

\begin{eqnarray}
&&
\mathcal{C}_c  K^{a_1 \ldots a_p}{}_{b_1 \ldots b_q}
\equiv D_c^{0,(p,q)} K^{a_1 \ldots a_p}{}_{b_1 \ldots b_q}
\equiv \nabla_c K^{a_1 \ldots a_p}{}_{b_1 \ldots b_q}
+\nonumber\\
&&
\frac{2}{1-D}\sum_{j=1}^{p}\!\left[
-\xi^{p}{}_{pc}\,\delta^{a_j}{}_{q_j}
-\xi^{p}{}_{pq_j}\,\delta^{a_j}{}_{c}
+g_{q_j c}\,g^{a_j d}\,\xi^{p}{}_{pd}
\right]
K^{a_1 \ldots a_{j-1} q_{j} a_{j+1} \ldots a_p}{}_{b_1 \ldots b_q}
+\nonumber\\
&&
\frac{2}{1-D}\sum_{j=1}^{q}\!\left[
\xi^{p}{}_{pc}\,\delta^{r_j}{}_{b_j}
+\xi^{p}{}_{pb_j}\,\delta^{r_j}{}_{c}
-g_{b_j c}\,g^{r_j d}\,\xi^{p}{}_{pd}
\right]
K^{a_1 \ldots a_p}{}_{b_1 \ldots b_{j-1}r_{j}b_{j+1} \ldots b_q}.
\label{eq:c-connection-flat-space}
\end{eqnarray}
If $D=3$ the Weyl tensor is identically zero and therefore
the analysis performed in this section can be adapted to that situation.

\subsection{The Robinson--Trautmann class of solutions}
\label{subsec:rt}
The Robinson--Trautmann (RT) spacetimes form a distinguished class of exact solutions
of Einstein's field equations describing radiative, asymptotically flat geometries.
They were originally introduced by Robinson and Trautman
\cite{robinsontrautman1960} and subsequently developed as a framework for the
description of gravitational radiation in the Bondi sense
\cite{Trautman1962}. These spacetimes are characterized by the existence of a null
geodesic congruence that is shear-free, twist-free, and expanding. This geometric
restriction leads to a remarkable simplification of the field equations and allows
for a rigorous treatment of nonlinear gravitational radiation.
In real coordinates $(u,r,x,y)$, the Robinson--Trautmann metric can be written as
\begin{equation}
ds^2
=
-2H(u,r,x,y)\,du^2
-2\,du\,dr
+2r^2 P^{-2}(u,x,y)\,(dx^2+dy^2),
\label{eq:RTmetric-real}
\end{equation}
where $u$ is a retarded time coordinate, $r$ is an affine parameter along outgoing
null geodesics, and $(x,y)$ are stereographic coordinates on the angular sections.
The function $P(u,x,y)$ determines the intrinsic geometry of the null hypersurfaces
$u=\mathrm{const}$.
If the above spacetime admits a geodesic, shear-free, twist-free
but expanding null congruence then the function $H$ is
given by
\begin{equation}
H
=
-r\,\partial_u \ln P
+ P^2\,(\partial_x^2+\partial_y^2)\ln P
-\frac{m(u)}{r}
-\frac{\Lambda}{6}\,r^2,
\label{eq:HRT-real}
\end{equation}
(see, for example, Refs.~\cite{Stephani2003,GriffithsPodolsky2009} for details).

The Robinson--Trautmann class includes important special cases such as the
Schwarzschild solution, as well as radiative black hole
spacetimes with time-dependent multipole moments. Owing to their exact nature and
their clear geometric and physical interpretation, RT solutions provide a valuable
testing ground for analytical, geometric, and numerical studies of gravitational
radiation \cite{GriffithsPodolsky2009}.

In this section we are going to apply Theorem \ref{theo:charact-einstein-spaces}
to find out under which conditions \eqref{eq:RTmetric-real} is conformal to
an Einstein space.
First of all, we need to compute the Weyl endomorphism
$C^{ab}{}_{cd}$ and from it its inverse $W^{ab}{}_{cd}$.
We will restrict ourselves to a situation in which
$H=H(u,r)$, $P=P(u)$.
To work out these computations, we define a vector bundle
whose standard fibre is a 6-dimensional vector space. In this
vector bundle these endomorphisms are naturally 1-contravariant, 1-
covariant tensors
(hence we name it the \emph{endomorphism vector bundle}).
Indices in the endomorphism vector bundle are denoted by
capital latin letters $A$, $B$,
$C$,  . . . and we have the relations
\begin{equation}
C_A{}^{B}=C^{ab}{}_{cd}\sigma^B{}_{ab}\tilde{\sigma}_A{}^{cd},\quad
W_A{}^{B}=W^{cd}{}_{ab}\sigma^B{}_{cd}\tilde{\sigma}_A{}^{ab}.\label{eq:cplus-end-rt}
\end{equation}
Here $\sigma^A{}_{ab}$, $\tilde{\sigma}_A{}^{ab}$ are
\emph{soldering forms}
that verify the algebraic properties

\begin{eqnarray}
&&
\sigma^A{}_{[ab]}=\sigma^A{}_{ab},\quad
\tilde{\sigma}_A{}^{[ab]}=\tilde{\sigma}_A{}^{ab}\\
&&
\delta_A{}^{B} =
\frac{1}{2}
\left(
-\delta_d{}^{a}\,\delta_b{}^{c}
+ \delta_b{}^{a}\,\delta_d{}^{c}
\right)
\tilde{\sigma}_A{}^{bd}\,\sigma_{ac}{}^{B},\quad
\frac{1}{2}
\left(
\delta_a{}^{b}\,\delta_c{}^{d}
-\delta_a{}^{d}\,\delta_c{}^{b}
\right)=
\tilde{\sigma}_A{}^{bd}\,\sigma_{ac}{}^{A}.
\end{eqnarray}
These properties do not uniquely
fix the soldering forms. If we choose a frame
$E\equiv\{e^1,e^2,e^3,e^4,e^5,e^6\}$
in the endomorphism
bundle and use the coordinate frame in the spacetime tangent
bundle, we have that a possible choice is
\begin{eqnarray}
\sigma^{1}_{u r} = 1,\;
\sigma^{2}_{u x^1} = 1,\;
\sigma^{3}_{u x^2} = 1,\;
\sigma^{4}_{r x^1} = 1,\;
\sigma^{5}_{r x^2} = 1,\;
\sigma^{6}_{x^1 x^2} = 1,
\\[1ex]
\tilde{\sigma}^{1}_{u r} = \frac{1}{2},\;
\tilde{\sigma}^{2}_{u x^1} = \frac{1}{2},\;
\tilde{\sigma}^{3}_{u x^2} = \frac{1}{2},\;
\tilde{\sigma}^{4}_{r x^1} = \frac{1}{2},\;
\tilde{\sigma}^{5}_{r x^2} = \frac{1}{2},\;
\tilde{\sigma}^{6}_{x^1 x^2} = \frac{1}{2}.
\label{eq:soldering-rt}
\end{eqnarray}
Now the practical procedure is to
start from the components of $C^{ab}{}_{cd}$
in the coordinate frame,
use \eqref{eq:cplus-end-rt}
and \eqref{eq:soldering-rt}
to compute $C_A{}^B$ in the frame $E$, use
a suitable algorithm to compute $W_A{}^B$ from
$C_A{}^B$ and then use again \eqref{eq:cplus-end-rt},
\eqref{eq:soldering-rt} to compute the components
of $(W^+)_{ab}{}^{cd}$ in the coordinate frame.
We display the values of $C_A{}^B$ and
$W_A{}^B$:
\begin{equation}
C_A{}^B=
\frac{2H - 2r\,H_{r} + r^{2} H_{rr}}{3 r^{2}}\,
\operatorname{diag}(2,1,1,1,1,2),\quad
W_A{}^B
=
\frac{3 r^{2}}{2H - 2r\,H_{r} + r^{2} H_{rr}}\,
\operatorname{diag}\!\left(\tfrac12,1,1,1,1,\tfrac12\right).
\end{equation}

Using \eqref{eq:cplus-end-rt} we can compute
$W^{ab}{}_{cd}$ and combining this result with
the Schouten tensor \eqref{schouten tensor definition} and \eqref{eq:define-Lambda}
yields
\begin{equation}
\Lambda
=
-\frac{1}{r}\left(
1+\frac{r^{3} H_{rrr}}{\D}
\right)\,dr
+
\left(
\frac{3P'}{P}
-
\frac{\D_{u}}{\D}
\right)\,du,
\label{eq:lambda-rt}
\end{equation}
where we introduced $\D$ as follows
\begin{equation}
\D\equiv 2H - 2r\,H_{r} + r^{2} H_{rr}.
\end{equation}
Now we can apply Theorem \ref{theo:charact-einstein-spaces}.
Eq. \eqref{eq:einstein-manifold} reduces to a set
in which one of the equations is
\begin{equation}
\frac{r\bigl(4 H_{rrr} + r H_{rrrr}\bigr)}{\D} = 0.
\end{equation}
This expression can be integrated yielding
\begin{equation}
H(u,r)
=
\frac{1}{2}\,Q(u)\,r^{2}
+
r\,R(u)
-
\frac{T(u)}{6\,r}
+ U(u),
\label{eq:value-of-H}
\end{equation}
where $Q$, $R$, $T$, $U$ are arbitrary functions of $u$.
This last equations enables us to write \eqref{eq:einstein-manifold}
in terms of $Q$, $R$, $T$, $U$. The result reduces to the
following two equations
\begin{equation}
\frac{P'}{P} = \frac{R T + 2U^{2} + T'}{2T}.
\label{eq:p-prime}
\end{equation}

\begin{equation}
\begin{aligned}
0 ={}&
60 r^{2}(-T + 2rU)\,P'^{2}
+ 2P\Bigl(\,P'\Bigl(
6Q r^{3}T + T^{2} - 6rTU + 24 r^{2}U^{2}
\\
&\qquad\qquad
+ 6r^{2}R(T + 2rU) + 18 r^{2}T' - 36 r^{3}U'
\Bigr)
+ 6r^{2}(T - 2rU)\,P''\Bigl)
\\
&\quad
+ P^{2}\Bigl(
6r^{2}R^2 T - 2TU^{2} + 12 rU^{3}
- 3 r^{3}T Q' - 6 r^{2}T R'
\\
&\qquad
+ 3Q r^{3}(RT + 2U^{2} - T')
- T T' + 6 rU T' - 24 r^{2}U U'
\\
&\qquad
- R\bigl(T^{2} - 6 rTU + 12 r^{2}(-U^{2} + rU')\bigr)
- 6 r^{2}T'' + 12 r^{3}U''
\Bigr),
\end{aligned}
\label{eq:no-p-prime}
\end{equation}
where \(P\neq 0\), \(T\neq 0\) was assumed.
Using \eqref{eq:p-prime} to compute the second derivative of \(P\)
and replacing the result into
\eqref{eq:no-p-prime} we obtain
\begin{equation}
\begin{aligned}
0 ={}&
3 Q R T^{3}
+ 6 Q T^{2} U^{2}
+ Q T^{2} T'
+ 12 R^{2} T^{2} U
-16 R T^{2} U'
+ 40 R T U^{3}
+ 20 R T U T'
\\[0.3em]
&{}- T^{3} Q'
- 4 T^{2} U R'
+ 4 T^{2} U''
- 40 T U^{2} U'
- 4 T U T''
- 12 T T' U'
\\[0.3em]
&{}+ 32 U^{5}
+ 40 U^{3} T'
+ 12 U (T')^{2}.
\end{aligned}
\label{eq:rtu}
\end{equation}
where the
$r$ dependence was factored out. Summarizing: Theorem \ref{theo:charact-einstein-spaces}
guarantees that \eqref{eq:RTmetric-real} is conformal to an Einstein space if
$H$ takes the form \eqref{eq:value-of-H}, $P\neq 0$, $T\neq 0$ and the functions
$Q$, $R$, $U$, $T$ are chosen as solutions of \eqref{eq:rtu}. To illustrate that
at least a solution exists, take the following values
\begin{equation}
 P(u)=e^{\frac{\lambda u}{2}},\quad T(u)=e^{\lambda u},\quad
 Q=R=U=0.
\end{equation}
The metric becomes
\begin{equation}
 ds^2=-\frac{2e^{\lambda u}}{r}du^2-2dudv+2e^{-\lambda u}r^2(dx^2+dy^2).
\end{equation}
This is not an einstein space but according to Theorem \ref{theo:charact-einstein-spaces}
it must be conformal to an Einstein space.

\subsection{Plane fronted waves}
\label{subsec:ppwave}
Let us consider the 4-dimensional Lorentzian manifold
$(\mathcal{M}, g_{ab})$ whose metric tensor
is given by the line element
\begin{equation}
ds^2=H(u,x^1,x^2)du^2-2du dv+(dx^1)^2+(dx^2)^2.
\label{eq:pp-wave}
\end{equation}
Here, $-\infty < u, v, x^1, x^2<\infty$ are global coordinates
on $\mathcal{M}$ and $H:\mathbb{R}^3\rightarrow\mathbb{R}$
is a smooth function. As is well-known $(\mathcal{M}, g_{ab})$
is an Einstein manifold if and only if the function $H$
fulfills the differential equation
\begin{equation}
 H_{x_1x_1}+H_{x_2x_2}=k H,\quad k\in\mathbb{R}.
 \label{eq:pp-wave-vacuum}
\end{equation}
Spacetimes of the form \eqref{eq:pp-wave} belong to the class of Brinkmann metrics \cite{Brinkmann1925}
describing \emph{plane-fronted waves}.
In this section we address the problem of finding the necessary
and sufficient conditions under which \eqref{eq:pp-wave}
is \emph{conformal} to an Einstein space using the tools of
Theorem \ref{theo:charact-einstein-spaces-plus}.
First of all, we need to compute the Weyl endomorphism
$C^{ab}{}_{cd}$ and from it the pseudo-inverse $(W^+)_{ab}{}^{cd}$.
To work out these computations, we follow a procedure similar to the one in
the previous subsection. In this case the components of $\sigma^A{}_{ab}$ and
$\tilde{\sigma}^A{}_{ab}$ are
\begin{eqnarray}
\sigma^{1}_{u v} = 1,\;
\sigma^{2}_{u x^1} = 1,\;
\sigma^{3}_{u x^2} = 1,\;
\sigma^{4}_{v x^1} = 1,\;
\sigma^{5}_{v x^2} = 1,\;
\sigma^{6}_{x^1 x^2} = 1,
\\[1ex]
\tilde{\sigma}^{1}_{u v} = \frac{1}{2},\;
\tilde{\sigma}^{2}_{u x^1} = \frac{1}{2},\;
\tilde{\sigma}^{3}_{u x^2} = \frac{1}{2},\;
\tilde{\sigma}^{4}_{v x^1} = \frac{1}{2},\;
\tilde{\sigma}^{5}_{v x^2} = \frac{1}{2},\;
\tilde{\sigma}^{6}_{x^1 x^2} = \frac{1}{2}.
\label{eq:soldering}
\end{eqnarray}
Now, as in subsection \ref{subsec:rt}, the practical procedure is to
start from the components of $C^{ab}{}_{cd}$
in the coordinate frame,
use \eqref{eq:cplus-end-rt}
and \eqref{eq:soldering}
to compute $C_A{}^B$ in the frame $E$, use
a suitable algorithm to compute $(W^+)_A{}^B$ from
$C_A{}^B$ and then use again \eqref{eq:cplus-end-rt},
\eqref{eq:soldering} to compute the components
of $(W^+)_{ab}{}^{cd}$ in the coordinate frame.
We display the values of $C_A{}^B$ and
$(W^+)_A{}^B$:

\begin{equation}
C_A{}^B=
\begin{pmatrix}
0 & 0 & 0 & 0 & 0 & 0\\
0 & 0 & 0 &
\dfrac{1}{2}\!\left(-H_{x^2 x^2}+H_{x^1 x^1}\right) &
H_{x^1 x^2} & 0\\
0 & 0 & 0 &
H_{x^1 x^2} &
\dfrac{1}{2}\!\left(H_{x^2 x^2}-H_{x^1 x^1}\right) & 0\\
0 & 0 & 0 & 0 & 0 & 0\\
0 & 0 & 0 & 0 & 0 & 0\\
0 & 0 & 0 & 0 & 0 & 0
\end{pmatrix},
\label{eq:weyl-PPwave}
\end{equation}

\begin{equation}
(W^+)_A{}^B=
\begin{pmatrix}
0 & 0 & 0 & 0 & 0 & 0 \\[1ex]
0 & 0 & 0 & 0 & 0 & 0 \\[1ex]
0 & 0 & 0 &
\dfrac{2\!\left(H_{x^2 x^2}-H_{x^1 x^1}\right)}
{4H_{x^1 x^2}^2+\left(H_{x^2 x^2}-H_{x^1 x^1}\right)^2}
&
\dfrac{4H_{x^1 x^2}}
{4H_{x^1 x^2}^2+\left(H_{x^2 x^2}-H_{x^1 x^1}\right)^2}
& 0 \\[3ex]
0 & 0 & 0 &
\dfrac{4H_{x^1 x^2}}
{4H_{x^1 x^2}^2+\left(H_{x^2 x^2}-H_{x^1 x^1}\right)^2}
&
\dfrac{2\!\left(H_{x^2 x^2}-H_{x^1 x^1}\right)}
{4H_{x^1 x^2}^2+\left(H_{x^2 x^2}-H_{x^1 x^1}\right)^2}
& 0 \\[3ex]
0 & 0 & 0 & 0 & 0 & 0 \\[1ex]
0 & 0 & 0 & 0 & 0 & 0
\end{pmatrix}.
\label{eq:weyl-plus-PPwave}
\end{equation}
As explained, using these matrices
we compute the components
of $(W^+)_{ab}{}^{cd}$ in the coordinate frame.
An important observation is that $(W^+)_{ab}{}^{cd}$
is trace-free and by construction,
antisymmetric in $ab$ and $cd$. However it does not have
all the symmetries of the Weyl tensor. For example
$$
W^+_{abcd}\neq W^+_{cdab},
$$
as can be checked by an explicit computation.

We have now all the ingredients to obtain all
the conditions of Theorem \ref{theo:charact-einstein-spaces-plus}. The first step is to impose
\eqref{eq:charact-einstein-spaces-sing} as explained
in Proposition \ref{prop:charact-einstein-spaces-sing}.
This yields the conditions
\begin{equation}
 H_{x^1x^2x^2}+H_{x^1x^1x^1}=0,\quad
 H_{x^2x^2x^2}+H_{x^1x^1x^2}=0.
\label{eq:necessary-ppwave}
\end{equation}
Next, using \eqref{eq:weyl-PPwave}-\eqref{eq:weyl-plus-PPwave} we
can compute the 1-form $\Lambda^{\xi}_a$ in the coordinates
of \eqref{eq:pp-wave}. The result is
\begin{equation}
\begin{aligned}
&\Lambda^\xi_adx^a=\frac{2}{3}\,dv\,\xi^{u}{}_{uv}
-\frac{2}{3}\,du\!
\left(
\xi^{v}{}_{uv}+\xi^{x^1}{}_{u x^1}+\xi^{x^2}{}_{u x^2}
\right)
\\[1ex]
&\quad
+\frac{1}{3}\,dx^2\Bigg(
2\,\xi^{u}{}_{u x^2}
+\frac{1}{\Delta}
\Big[
2\,\xi^{x^1}{}_{x^1 x^2}\,\Delta
-\left(H_{x^2 x^2}-H_{x^1 x^1}\right)
\left(H_{x^2 x^2 x^2}+H_{x^1 x^1 x^2}\right)
\\
&\hspace{7.5em}
-2H_{x^1 x^2}
\left(H_{x^1 x^2 x^2}+H_{x^1 x^1 x^1}\right)
\Big]
\Bigg)
\\[2ex]
&\quad
+\frac{1}{3}\,dx^1\Bigg(
2\,\xi^{u}{}_{u x^1}
+\frac{1}{\Delta}
\Big[
-2\,\xi^{x^2}{}_{x^1 x^2}\,\Delta
-2H_{x^1 x^2}
\left(H_{x^2 x^2 x^2}+H_{x^1 x^1 x^2}\right)
\\
&\hspace{7.5em}
+\left(H_{x^2 x^2}-H_{x^1 x^1}\right)
\left(H_{x^1 x^2 x^2}+H_{x^1 x^1 x^1}\right)
\Big]
\Bigg),
\end{aligned}
\end{equation}
Where we set $
\Delta\equiv
4\,H_{x^1 x^2}^{\,2}
+\bigl(H_{x^2 x^2}-H_{x^1 x^1}\bigr)^{2}$. This can be
further simplified with the aid of \eqref{eq:necessary-ppwave}. The
result is
\begin{equation}
\Lambda^\xi_adx^a=
 \frac{2}{3}\,dv\,\xi^{u}{}_{u v}
+\frac{2}{3}\,dx^2\!\left(\xi^{u}{}_{u x^2} + \xi^{x^1}{}_{x^1 x^2} \right)
-\frac{2}{3}\,du\!\left( \xi^{v}{}_{u v} + \xi^{x^1}{}_{u x^1} + \xi^{x^2}{}_{u x^2} \right)
+\frac{2}{3}\,dx^1\!\left(\xi^{u}{}_{u x^1} - \xi^{x^2}{}_{x^1 x^2} \right)
\end{equation}
An explicit computation reveals that
$\boldsymbol\Lambda^\xi$
is closed if and only if

\begin{equation}
\begin{aligned}
\partial_v \xi^{x^2}_{u x^2}
&= -\,\partial_u \xi^{u}_{u v}
   - \partial_v \xi^{v}_{u v}
   - \partial_v \xi^{x^1}_{u x^1},
\\[1ex]
\partial_{x^1} \xi^{u}_{u v}
&= \partial_v \xi^{u}_{u x^1}
   - \partial_v \xi^{x^2}_{x^1 x^2},
\\[2ex]
\partial_{x^1} \xi^{x^2}_{u x^2}
&= -\,\partial_u \xi^{u}_{u x^1}
   + \partial_u \xi^{x^2}_{x^1 x^2}
   - \partial_{x^1} \xi^{v}_{u v}
   - \partial_{x^1} \xi^{x^1}_{u x^1},
\\[1ex]
\partial_{x^2} \xi^{u}_{u v}
&= \partial_v \xi^{u}_{u x^2}
   + \partial_v \xi^{x^1}_{x^1 x^2},
\\[2ex]
\partial_{x^2} \xi^{x^2}_{u x^2}
&= -\,\partial_u \xi^{u}_{u x^2}
   - \partial_u \xi^{x^1}_{x^1 x^2}
   - \partial_{x^2} \xi^{v}_{u v}
   - \partial_{x^2} \xi^{x^1}_{u x^1},
\\[1ex]
\partial_{x^2} \xi^{x^2}_{x^1 x^2}
&= -\,\partial_{x^1} \xi^{u}_{u x^2}
   - \partial_{x^1} \xi^{x^1}_{x^1 x^2}
   + \partial_{x^2} \xi^{u}_{u x^1}.
\end{aligned}
\label{eq:lambda-xi-closed}
\end{equation}
Assuming the previous set of equations
we impose the conditions of Theorem \ref{theo:charact-einstein-spaces-plus}.
Since $\Lambda^\xi_a$ is explicitly known, we can compute
$\Gamma[\mathcal{C}, \nabla]^{c}{}_{ab}$ using  \eqref{eq:connection-C-nabla}
and $\mathcal{R}^\xi_{ab}$ by means of \eqref{eq:ricci-tensor-c-conection}. One
gets that $\mathcal{R}^\xi_{ab}$ is symmetric so the first of
\eqref{eq:einstein-manifold-plus} is automatically fulfilled whereas the second
condition (trace-free part of $\mathcal{R}^\xi_{ab}$ equal to zero) is equivalent
to
\begin{eqnarray}
\partial_u \xi^{x^1}_{x^1 x^2} &=&
\frac{1}{6}\Big(
-4\big(\xi^{u}_{u x^2}+\xi^{x^1}_{x^1 x^2}\big)
\big(\xi^{v}_{u v}+\xi^{x^1}_{u x^1}+\xi^{x^2}_{u x^2}\big)
-6\,\partial_u \xi^{u}_{u x^2}
-3\,\xi^{u}_{u v}\,H_{x^2}
\Big),
\\[1ex]
\partial_u \xi^{x^2}_{u x^2} &=&
-\frac{2}{3}\big(\xi^{v}_{u v}+\xi^{x^1}_{u x^1}+\xi^{x^2}_{u x^2}\big)^2
+\frac{1}{3}H
\Big(
2\,\xi^{u}_{u v}
\big(\xi^{v}_{u v}+\xi^{x^1}_{u x^1}+\xi^{x^2}_{u x^2}\big)
+3\,\partial_u \xi^{u}_{u v}
\Big)
\nonumber\\
&&
+\frac{1}{8}\Big(
-8\,\partial_u \xi^{v}_{u v}
-8\,\partial_u \xi^{x^1}_{u x^1}
+4\big(\xi^{u}_{u x^2}+\xi^{x^1}_{x^1 x^2}\big)H_{x^2}
+3H_{x^2 x^2}
\nonumber\\
&&
+4\big(\xi^{u}_{u x^1}-\xi^{x^2}_{x^1 x^2}\big)H_{x^1}
+3H_{x^1 x^1}
+4\,\xi^{u}_{u v}H_{u}
\Big),
\\[1ex]
\partial_u \xi^{x^2}_{x^1 x^2} &=&
\frac{2}{3}
\big(\xi^{v}_{u v}+\xi^{x^1}_{u x^1}+\xi^{x^2}_{u x^2}\big)
\big(\xi^{u}_{u x^1}-\xi^{x^2}_{x^1 x^2}\big)
+\partial_u \xi^{u}_{u x^1}
+\frac{1}{2}\,\xi^{u}_{u v}H_{x^1},
\\[1ex]
\partial_v \xi^{u}_{u v} &=&
\frac{2}{3}\,\xi^{u}_{u v},
\\[1ex]
\partial_v \xi^{x^1}_{x^1 x^2} &=&
\frac{2}{3}\,\xi^{u}_{u v}
\big(\xi^{u}_{u x^2}+\xi^{x^1}_{x^1 x^2}\big)
-\partial_v \xi^{u}_{u x^2},
\\[1ex]
\partial_v \xi^{x^2}_{x^1 x^2} &=&
\frac{2}{3}\,\xi^{u}_{u v}
\big(-\xi^{u}_{u x^1}+\xi^{x^2}_{x^1 x^2}\big)
+\partial_v \xi^{u}_{u x^1},
\\[1ex]
\partial_{x^1} \xi^{x^1}_{x^1 x^2} &=&
\frac{2}{3}
\big(\xi^{u}_{u x^2}+\xi^{x^1}_{x^1 x^2}\big)
\big(\xi^{u}_{u x^1}-\xi^{x^2}_{x^1 x^2}\big)
-\partial_{x^1} \xi^{u}_{u x^2},
\\[1ex]
\partial_{x^1} \xi^{x^2}_{x^1 x^2} &=&
-\frac{2}{3}\Big(
-\xi^{u}_{u v}
\big(\xi^{v}_{u v}+\xi^{x^1}_{u x^1}+\xi^{x^2}_{u x^2}\big)
+\big(\xi^{u}_{u x^1}-\xi^{x^2}_{x^1 x^2}\big)^2
\Big)
\nonumber\\
&&
+\partial_u \xi^{u}_{u v}
+\partial_{x^1} \xi^{u}_{u x^1},
\\[1ex]
\partial_{x^2} \xi^{x^1}_{x^1 x^2} &=&
\frac{2}{3}\Big(
\big(\xi^{u}_{u x^2}+\xi^{x^1}_{x^1 x^2}\big)^2
-\xi^{u}_{u v}
\big(\xi^{v}_{u v}+\xi^{x^1}_{u x^1}+\xi^{x^2}_{u x^2}\big)
\Big)
-\partial_u \xi^{u}_{u v}
-\partial_{x^2} \xi^{u}_{u x^2}.
\label{eq:xi-system}
\end{eqnarray}
To assess whether the system admits integrability conditions involving
only the background function \(H\), we regard the equations as a
first--order PDE system for the unknown fields \(\xi\),
with \(H(u,x^1,x^2)\) treated as prescribed data. Formal integrability
is then analyzed by computing compatibility conditions obtained from
commutators of mixed derivatives and by prolonging the system as
necessary. The resulting relations are systematically simplified with
the aim of eliminating all occurrences of \(\xi\) and their
derivatives. If, after this elimination process, a nontrivial relation
remains that depends only on \(H\) and its derivatives, such a
relation constitutes an integrability condition on \(H\) that
should be appended to \eqref{eq:necessary-ppwave}.
If no such a relation exists then the conditions of
Theorem \ref{theo:charact-einstein-spaces-plus} reduce to just \eqref{eq:necessary-ppwave}
which would then become necessary and sufficient conditions for \eqref{eq:pp-wave}
to be conformal to an Einstein space.

Note that \eqref{eq:necessary-ppwave} is consistent with previous work
because in \cite{BrozosVazquez2019IsotropicQE} it is shown that
$\nabla_aC^{a}{}_{bcd}=0$ iff \eqref{eq:pp-wave} is conformal to an Einstein
space. But a straightforward computation shows that
$\nabla_aC^{a}{}_{bcd}=0$ is actually equivalent to \eqref{eq:necessary-ppwave}.
Hence the result of \cite{BrozosVazquez2019IsotropicQE} means that
\eqref{eq:xi-system} entails no integrability conditions over $H(u,x^1,x^2)$.

A notable example
is the trivial solution of the equations \eqref{eq:xi-system} in which all \(\xi^a{}_{bc}\) vanish identically;
in this case, the above system reduces to a single
nontrivial condition on \(H\), namely the Laplace equation
\begin{equation}
 H_{x^1x^1}+H_{x^2x^2}=0,
\label{eq:laplace}
\end{equation}
showing that \(H\) must be harmonic in the transverse variables.
This correspond to the condition that the plane fronted wave is an Einstein
space with zero cosmological constant. Therefore by imposing \(\xi^a{}_{bc}=0\) we recover
the standard Einstein vacuum solution condition (note that in this case
\eqref{eq:necessary-ppwave} is a consequence of \eqref{eq:laplace}).

\section{Conclusions}

In this work, we have successfully generalized the construction of the $\mathcal{C}$-connection,
originally introduced in \cite{garciaparrado2024einstein} for four-dimensional
spacetimes and for an invertible Weyl endomorphism, to arbitrary dimension $D$ and to
cases where the Weyl endomorphism is singular.

The core of our procedure relied on the algebraic properties of the Weyl
tensor regarded as an endomorphism on the bundle of $2$-forms.
In Lemma \ref{lemm:fund} we obtained a fundamental
formula which relates $\Upsilon_{a}$ with metric concomitants
of the physical and unphysical spaces. If the
Weyl endomorphism is invertible, only metric concomitants appear and if
it is not then one has an additional tensor $\xi^{a}{}_{bc}{}$ whose role is also studied.

As for the algebraic results, we introduced the conformally covariant derivative operators $D_{a}^{s(p,q)}$
that depend on the conformal weight $s$ and the tensor character $p,q$
and defined the ${\mathcal C}$-connection, $\mathcal{C}_{a}$, in the particular case in which
$s=0$.
We also introduced its singular counterpart $\mathcal{C}_{a}^{\xi}$
for a non-invertible Weyl endomorphism. All this generalizes work done in \cite{garciaparrado2024einstein}

Furthermore, we applied this formalism to the problem of characterizing conformal Einstein spaces. We provided necessary and
sufficient conditions for a pseudo-Riemannian manifold to be conformally related to an Einstein space. These conditions,
encapsulated in Theorems \ref{theo:charact-einstein-spaces} and \ref{theo:charact-einstein-spaces-plus}, are expressed
entirely in terms of concomitants of the unphysical metric and differential properties. Theorem
\ref{theo:charact-einstein-spaces} is written entirely in terms of metric concomitants of the unphysical metric and
therefore it enables us to decide algorithmically whether a given pseudo-Riemannian metric meeting the Theorem's
hypotheses is conformal to an Einstein space.

Finally, we explored the practical implications of our results through
examples, such as conformally flat, Robinson-Trautman-like and
plane fronted waves.
These demonstrate the robustness of our approach.
This formalism opens new ways for studying conformal
symmetries and geometric invariants in higher-dimensional theories where the
algebraic classification of the Weyl tensor plays a significant role.

\section*{Acknowledgements}
We thank Prof. José M. Senovilla for reading the
manuscript and useful comments.
Supported by Spanish MICINN Project No. PID2021-126217NB-I00.

\appendix

\section{The Moore-Penrose pseudo-inverse}
\label{app:pseudo-inverse}
In this appendix we review known facts about the Moore-Penrose pseudo-inverse that are used in
the paper. We will make our exposition as self-contained as possible. More details can be found
in e.g. \cite{Bajo2021}.

\begin{definition}[Moore-Penrose pseudoinverse]
	Let $A \in M_{m \times n}(\mathbb{C})$, $m, n\in\mathbb{N}$.
	The Moore-Penrose pseudoinverse of $A$, denoted by
	$A^+ \in M_{n \times m}(\mathbb{C})$,
	is the unique matrix that satisfies the following equations:
	\begin{align}
		&1.\quad A A^{+} A = A \label{eq:pseudo-1}\\
		&2.\quad A^{+} A A^{+} = A^{+} \label{eq:pseudo-2}\\
		&3.\quad (A A^{+})^{*} = A A^{+} \label{eq:pseudo-3}\\
		&4.\quad (A^{+} A)^{*} = A^{+} A, \label{eq:pseudo-4}
	\end{align}

	where $(\cdot)^{*}$ denotes the conjugate transpose.

	\label{def:moore-penrose-pseudoinverse}
\end{definition}

If $A$ is invertible, then $A^+=A^{-1}$ and the usual notion of inverse
is recovered. Elementary properties of the pseudo-inverse
are  $(\lambda A)^+=A^+ / \lambda$, $\lambda\in\mathbb{R}$ and
$(A^*)^+=(A^+)^*$.
For us the main interest of the pseudo-inverse is that it
can be used to express the general solution of any linear system written in matrix
form
\begin{equation}
 A X=B,
 \label{eq:linear-system}
\end{equation}
where $X\in M_{n\times 1}(\mathbb{C})$ is the unknown and
$A\in M_{m \times n}(\mathbb{C})$, $B\in M_{m\times 1}(\mathbb{C})$
are given matrices.
\begin{theorem}[See \cite{James_1978}]
	The linear system $AX=B$ has a solution if and only if $AA^+B=B$.
	Moreover, in that case the general solution is given by
		\begin{equation}
		X=A^+B+(I-A^+A)w
		\label{eq:pseudoinverse-sol}
	\end{equation}
where $I$ is the identity matrix in $M_n({\mathbb{C}})$ and
$w\in M_{n\times 1}(\mathbb{C})$ is arbitrary.

	\label{theo:compatibility-condition}
\end{theorem}

\begin{proof}
	 To show the ``only if'' part we will multiply both sides of
	 the equation $AX=B$ by $AA^+$ getting
	 \begin{equation}
	 	AA^+AX=AA^+B.
	 \end{equation}

	 Using the pseudo-inverse property  $AA^+A=A$, the previous equation transforms into
	 $AX=AA^+B$ which is equivalent to $B=AA^+B$.

	 To show the ''if'' part, assume that $AA^+B=B$ holds. Then it is clear that
	 $A^+B$ is a solution of \eqref{eq:linear-system}.

	Next we show that \eqref{eq:pseudoinverse-sol} is indeed the general
	solution of \eqref{eq:linear-system}. Elementary algebra tells us that
	the general solution can be written as $X=X_p+X_k$ where $X_p$ is a particular
	solution of \eqref{eq:linear-system} and $X_k$ is any element in
	$\ker(A)$. The previous considerations show that $X_p=A^+A$ is a particular solution.
	Hence to finish the proof we need to show that any $X_k$ can be
	written as $X_k=(I-A^+A)w$ with $w\in M_{n\times 1}(\mathbb{C})$ arbitrary or in other
	words $\ker(A)=\operatorname{Im}(I-A^+A)$. Let $Y\in \operatorname{Im}(I-A^+A)$.
	Then $Y=(I-A^+A)w$ and
	\begin{equation}
		AY=A(I-A^+A)w=(AI-AA^+A)w=(A-A)w=0,
	\end{equation}
thus $\operatorname{Im}(I-A^+A) \subseteq \ker(A)$. Conversely pick $k\in\ker(A)$.
Then
	\begin{equation}
		(I-A^+A)k=Ik-A^+Ak=I k =k.
	\end{equation}

	So, $\ker(A) \subseteq \operatorname{Im}(k) $ and we conclude
	that $ \ker(A) = \operatorname{Im}(k)$ as desired.
\end{proof}

Let $A$ be a matrix (which can be non-invertible), then we denote $A^+$ its Moore-Penrose pseudoinverse which can be used to solve any linear system $AX=B$ by

Using the index notation we can write \eqref{eq:pseudoinverse-sol} as

\begin{equation}
	X^{I}=(A^+)^{I}{}_{J}{}B^{J}+(\delta^{I}{}_{J}-(A^+)^{I}{}_{K}{}A^{K}{}_{J}{})w^{J}.
	\label{eq:pseudo-inverse-indices}
\end{equation}

Now, we will see hoy it adapts to the fields of endomorphisms introduced in \eqref{eq:weyl-endomorphism}, where we regarded the Weyl tensor $C_{abcd}$ as a linear map (endomorphism) 

\begin{equation} \psi:\ \Omega^{2}(\mathcal{M})\ \longrightarrow\ \Omega^{2}(\mathcal{M}),\qquad C_{ab}{}^{cd}\longrightarrow C_{I}{}^{cd}=\psi_{I}{}^{ab} C_{ab}{}^{cd}.
	\label{eq:endomorphism_field_weyl_def}
\end{equation}

This field of endomorphisms, which is a 6-dimensional space of bivectors, is defined to respect the pairs antisymmetries of the Weyl tensor, and the rest of the symmetries (double pairs symmetry, traceless and Bianchi identity) keep as a consequence of the definition of Weyl tensor. 

Now for the Moore-Penrose pseudoinverse of the Weyl tensor, it been defined (and used in our calculations) as a linear map from the field of endomorphisms, so $W^+\in\Omega^2(\mathcal{M})$. This make sure that the pairs antisymmetry can also be found in the pseudoinverse of the Weyl tensor.

Note that the application defined in \eqref{eq:endomorphism_field_weyl_def} is linear, so we can find its inverse and recover the tensor form of the Weyl tensor and its pseudoinverse, but from the last one, as it is defined from \ref{def:moore-penrose-pseudoinverse}, it does not represent the inverse of the Weyl tensor (unless it is invertible, where $W^+=C^{-1}$) and also, it is only guaranteed to have the the pairs antisymmetry found on the original Weyl tensor, as a consequence of the definition from the field of endomorphisms, but the double pairs antisymmetry, traceless and Bianchi identity properties are not guaranteed to be found in the pseudoinverse either seen as an endomorphism or as a tensor.


\end{document}